\documentclass[11p,reqno]{amsart}
\usepackage{amsmath}
\usepackage{amsfonts}
\usepackage{amssymb}
\usepackage{graphicx}
\usepackage{color}
\newtheorem{theorem}{Theorem}[section]
\newtheorem*{theorem*}{Theorem}

\newtheorem{corollary}[theorem]{Corollary}
\newtheorem{proposition}{Proposition}[section]

\numberwithin{equation}{section}

\begin{document}
	\title[An extension of Schur's theorem ]{A Schur's theorem via a monotonicity and the expansion module}

\author{Lei Ni}
\address{Lei Ni. Department of Mathematics, University of California, San Diego, La Jolla, CA 92093, USA}
\email{leni@ucsd.edu}


\subjclass[2010]{}

\begin{abstract} In this paper we  present a monotonicity which extends a classical theorem of A. Schur comparing the chord length of a convex plane curve with that of a space curve of smaller curvature. We also prove a Schur's Theorem for  spherical curves, which extends the Cauchy's Arm Lemma.
\end{abstract}

\maketitle

\section{Introduction}
For a convex curve $c(s): [0, L]\to \mathbb{R}^2$  and a smooth curve in $\tilde{c}(s): [0, L]\to \mathbb{R}^3$ of the same length (both parametrized by the arc-length), A. Schur's theorem \cite{Hopf} (Theorem A page 31, see also \cite{Chern}) asserts that {\it  if both curves are embedded, and the curvature of the space curve $\tilde{k}(s):= |\tilde{T}'|(s)$, where $\tilde{T}(s)={\tilde{c}}'(s)$ is the tangent vector,  is not greater than the curvature $k(s)$ of the convex curve, then $d_{\mathbb{R}^3}(\tilde{c}(0), \tilde{c}(L))\ge d_{\mathbb{R}^2}(c(0), c(L))$}. Here $c'(s)$ denotes $\frac{d}{ds} c(s)$. From the proof of \cite{Hopf} it is easy to see $\mathbb{R}^3$ can be replaced by $\mathbb{R}^{n+1}$ with any $n\ge 1$.

The theorem can be proven for curves whose tangents have  finite discontinuous jumps, and to the situation that the curvature of the smaller curve $\tilde{c}(s)$ is a curve in $\mathbb{R}^{n+1}$ for $n\ge 1$. For each $1\le j\le N-1$ at the point $c(s_j)$ the oriented turning angles (say counter-clock wisely) is measured by signed distance $\alpha_j:=d_{\mathbb{S}^n}(c'(s_{j}-), c'(s_{j}+))>0$ (here note that $\{c'(s)\}$ is viewed as points in a great circle $\mathbb{S}^1$ inside $\mathbb{S}^n$ and an orientation is given based on the region enclosed by $c(s)$ and the chord $\overline{c(0)c(L)}$ is on the left side of $c'(s)$). The turning angle $\tilde{\alpha}_j$ is measured simply by $\tilde{\alpha}_j=d_{\mathbb{S}^n}(\tilde{c}'(s_{j}-), \tilde{c}'(s_{j}+))$. In terms of the  generalization to curves with finite discontinuous points for the tangent, it assumes that there exists $\{s_j\}_{0\le j\le N}$ such that $0=s_0<s_1< \cdots <s_k < \cdots <s_{N}=L$ such that both $c(s)$ and $\tilde{c}(s)$ are regular embedded curves for $s\in (s_{j-1}, s_{j})$ for all $1\le j\le N$ satisfying $k(s)\ge \tilde{k}(s)$ and  above defined $\alpha_j$ and $\tilde{\alpha}_j$ satisfy that $\alpha_j\ge \tilde{\alpha}_j$ for all $1\le j\le N-1$. The convexity of $c(s)$ and the simpleness assumption imply that $\alpha_j\in (0, \pi)$ and
\begin{equation}\label{eq:1-1}
\sum_{j=1}^N \int_{s_{j-1}}^{s_j} k(s)\, ds +\sum_{j=1}^{N-1}\alpha_j \le 2\pi.
\end{equation}

This extension, together with some ingenious applications of the hinge's theorem, allows one to prove the famous Cauchy's Arm Lemma for geodesic arms in the unit sphere (consisting of continuous broken great/geodesic arcs with finite jumps of the tangents) in Lemma II  on the  pages 37--38 of \cite{Hopf}. The Lemma became famous due to that it had an incomplete/false proof by Cauchy originally \cite{Cauchy}. The corrected proof appeared in \cite{Alex, SR}. This spherical Cauchy's Arm Lemma can also be proved by an induction argument \cite{SZ}, whose idea in fact in part resembles the proof of  the smooth case to some degree. Note that this lemma of Cauchy plays a crucial role in the rigidity of convex polyhedra in $\mathbb{R}^3$, which finally was vastly generalized to  convex surfaces (convex bodies enclosed) as the famous Pogorelov monotypy theorem (cf.  \cite{Busemann} Section 21).

 The Schur's theorem also can be applied to prove the four-vertex theorem for convex plane curves,  besides implying a Theorem of H. A. Schwartz which asserts: {\it For any curve $c$ of length $L$ with curvature $k(s)\le 1/r$, let $C$ be the circle passing $c(0)$ and $c(L)$ of radius $r$, then $L$ is either not greater than the length of the lesser circular arc, or not less  than the length of the greater circular arc of $C$. } High dimensional (intrinsic) analogues of A. Schur's theorem include the Rauch's comparison theorem and  the Toponogov comparison theorem. The later  however has the limit of  requiring that the manifold with less curvature must be a space form of constant sectional curvature.

First we have the following slightly more general version of Schur's theorem in terms of a monotonicity.

 \begin{theorem}\label{thm:main1} Let $c: [0, L]\to \mathbb{R}^2$ be a simple piece wisely regular convex plane curve with curvature $k(s)\ge 0$ and finite many discontinuities for the tangent at $\{s_j\}_{j=1}^{N-1}$. Let $\tilde{c}: [0, L]\to \mathbb{R}^{n+1}$ ($n\ge 1$) be another simple  curve such that $\tilde{k}(s)=|\tilde{T}'|(s)\le k(s)$. Moreover we assume that the turning angles of $\{\alpha_j\}$ and $\tilde{\alpha}_j$ satisfy $\alpha_j\ge \tilde{\alpha}_j$.  Then for any $0\le s'<s''\le L$ there exists a linear  isometric map $\iota_{s', s''}: \mathbb{R}^2\to \mathbb{R}^{n+1}$ with $\iota_{s', s''}(0)=0$ such that
 $$
 I(s):=\langle \tilde{c}(s)-\iota_{s', s''}(c(s)),  \iota_{s', s''}(c(s'')-c(s'))\rangle
 $$
 is monotone non-decreasing for $s\in [s', s'']$, or equivalently
 \begin{equation}\label{eq:10}
 \langle \tilde{T}(s)-\iota_{s', s''}(T(s)), \iota_{s', s''}(c(s'')-c(s'))\rangle \ge 0,\quad  \forall\, s \in [s', s'']\setminus \cup_{k=1}^{N-1}\{s_j\}.
 \end{equation}
  \end{theorem}
There is a freedom of the linear isometry by a rotation fixing the vector $\iota_{s', s''}(c(s'')-c(s')$. As  $s'\to s''$, when $s''$ is a smooth point, the linear isometric embedding $\iota_{s', s''}$ converges to one  identifying $T(s)$ with $\tilde{T}(s)$ factoring this freedom.
\begin{corollary} Under the same assumption as in the theorem,  for any $s'\le s'_*<s''_{*}\le s''$,
 \begin{equation}\label{eq:101}
 \langle c(s''_*)-c(s'_*), c(s'')-c(s')\rangle\le \langle\tilde{c}(s''_*)-\tilde{c}(s'_*), \iota_{s', s''}(c(s'')-c(s'))\rangle.
 \end{equation}
 When $s'=s'_*$ and $s''=s''_*$ we have that
 \begin{equation}\label{eq:11}
 |c(s'')-c(s')|^2 \le \langle\tilde{c}(s'')-\tilde{c}(s'), \iota_{s', s''}(c(s'')-c(s'))\rangle.
 \end{equation}
 The equality holds if and only if $\iota (c(s))|_{[s', s'']}$ and $\tilde{c}(s)|_{[s', s'']}$ are the same.
\end{corollary}

The estimate (\ref{eq:11}) implies Schur's theorem by the Cauchy-Schwarz inequality applied to the right hand side of (\ref{eq:11}):
$$
|c(s'')-c(s')|\le |\tilde{c}(s'')-\tilde{c}(s')|, \quad \forall\,  0\le s'< s''\le L.
$$
 This extension allows  one to rephrase the result in terms of  the concept of the {\it expansion module} \cite{An, Ni-JMPA} of vector fields. If $X: \Omega\subset \mathbb{R}^{n+1} \to \mathbb{R}^{n+1}$ is a vector field defined on a convex domain, then the expansion module is a function of one variable $\psi(t)$ such that
$$
\langle X(y)-X(x), \frac{y-x}{|y-x|}\rangle \ge 2\psi\left(\frac{|x-y|}{2}\right).
$$
Since $\tilde{c}(s)$ and $c(s)$ are related via the parameter $s$, one may view $\tilde{c}$ as a related vector field defined over $\iota_{s', s''}(c(s))\in \mathbb{R}^{n+1}$. Now the estimate in Theorem \ref{thm:main1} simply asserts that the related vector fields $\tilde{c}(s)$ has an expansion  module function $\psi(t)=t$ with respect to the associated vector $\iota_{s', s''}(c(s))$.

From the above   connection between  the concept of {\it curvature} and the {\it expansion module} it is our  hope that a high dimensional Schur's theorem  could be discovered through the consideration involving the expansion module.

Given that Schur's theorem implies the Cauchy's Arm Lemma for the arms of great arcs in the unit sphere, a natural question is that if the spherical analogue of Schur's theorem still holds. Namely, given two embedded spherical curves $c(s)$ and $\tilde{c}(s)$ in the unit sphere $\mathbb{S}^2\subset \mathbb{R}^3$ parametrized by the arc-length $s\in [0, L]$ with $L\le \pi$. Assume that $c(s)$ is convex with geodesic curvature $k(s)\ge 0$, and  that the geodesic curvature of $\tilde{c}$ satisfies  $|\tilde{k}|(s)\le k(s)$. Does it still hold that $d_{\mathbb{S}^2}(c(0), c(L))\le d_{\mathbb{S}^2}(\tilde{c}(0),\tilde{c}(L))$? One could also allow the tangent of curves to have same amount of finite many jumps at $\{s_j\}$. In that case, at each $s_j$, the oriented angle $\alpha_j\in (0, \pi)$ and the turning angle $\tilde{\alpha}_j=d_{\mathbb{S}^2}(\tilde{T}(s_j-), \tilde{T}(s_j+))$ are assumed to satisfy that $\alpha_j\ge \tilde{\alpha}_j$ as in the previous case. The Cauchy's Arm Lemma in the sphere provides a positive answer in the special case where both curves have zero geodesic curvature for the smooth parts. Here we confirm this conjecture by proving

\begin{theorem}\label{thm:main2} The Schur's theorem holds for two curves in $\mathbb{S}^2\subset \mathbb{R}^3$ under the above configurations similar to that of Theorem \ref{thm:main1}. The equality holds if and only if the two curves are congruent by an isometry of $\mathbb{S}^2$.
\end{theorem}

The proof is by a construction of auxiliary curves via the cones over $c(s)$ and $\tilde{c}(s)$, with one of them being a convex plane curve, and then reduce it to the above generalized  Schur's theorem, namely \ref{thm:main1}.  For the reason of completeness we present the proof of Theorem \ref{thm:main1} with  details.
Note that Theorem \ref{thm:main2} generalizes the spherical Cauchy's Arm Lemma. It would be interesting to see if it plays any role in the proof of  Pogorelov's monotype theorem. There were extensions of  A. Schur's theorem in hyperbolic spaces \cite{Epstein} and in the Minkowski plane \cite{Lopez} earlier. It is also interesting to see if the method of this paper can be  used to simplify the argument in the former work via Theorem \ref{thm:21}. In the excellent exposition \cite{Arnold} one can find many other interesting topics and constructions on spherical curves and their relevance in the study of mechanics.

The generalization in Theorem \ref{thm:main1} allows a truly spherical version of Theorem \ref{thm:main1} in Theorem \ref{thm:32} of a comparison between a convex curve in $\mathbb{S}^2$ and a curve in $\mathbb{S}^n$, for any $n\ge 2$, with its geodesic curvature $\tilde{k}(s)\le k(s)$.

\section{Proof of Theorem \ref{thm:main1}}

We prove theorem and its corollary together. After a linear isometric embedding $\iota: \mathbb{R}^2\to \mathbb{R}^{n+1}$, which shall be specified later, we may consider the tangent $T(s)$ and $\tilde{T}(s)$ as two curves in $\mathbb{S}^n$. For the situation when the tangent $T(s)$ has a jump at $s_j$, the minimizing arc jointing $T(s_j-)$ and $T(s_j+)$ is also considered to be part of the image. They together form a part of a great circle which is denoted by $\operatorname{Image}(T(s))$. Namely we may view $T(s)$ as a set valued map into $\mathbb{S}^1\subset \mathbb{R}^2$.  First we choose a point $\mathcal{N}\in \operatorname{image}(T(s))$ as a {\it normalized tangent}. Without the loss of generality we assume $s'$ and $s''$ are two smooth points.

 The convexity of $c(s)$ implies that $\operatorname{image}(T(s))$ is  part of a great arc of $\mathbb{S}^1\subset \mathbb{R}^2$. If we parametrize it with angle $\theta(s)$ from the positive axis, $\theta (s'')\ge \theta(s')$ if $s''\ge s'$.   We first find a $s_*\in [s', s'']$ and a vector $\mathcal{N}\in T(s_*)$ such that it is a positive multiple of $c(s'')-c(s')$. When $s_*$ is a smooth point $T(s_*)$ is single valued. If $s_*=s_j$ for some $j$ we pick one value in the arc spanning from $T(s_*-)$ to $T(s_*+)$. To make the argument easy, we put  the plane curve $c(s)|_{[s', s'']}$ into a $xy$-plane so that the vector $c(s'')-c(s')$ is in the positive $x$-axis direction, and $c(s')$ and $c(s'')$ are on the $x$-axis. The curve $c(s)|_{(s', s'')}$ is inside the lower half plane $y\le 0$. The curve $c(s)$ can also be expressed as $(x(s), y(s))$ with $y(s)\le 0$. Note that  $y(s')=y(s'')=0$ (and $x(s')<x(s'')$). The continuous piece-wisely smooth function $y(s)$ attains a minimum somewhere at $s_*\in (s', s'')$. Without loss of generality we may assume that $y(s_*)<0$. If $s_*\ne s_j$, namely it is a smooth point, then $y'(s_*)=0$, which  implies that $T(s_*)=(x'(s_*), 0)$. Namely $T(s_*)$ is  parallel to the $x$-axis. By the convexity, the angle between $T(s)$ and the $x$-axis starts with  $\theta(s')\ge -\pi$, and monotonically increases into $\theta(s'')\le \pi$. Hence the angle between $T(s_*)$ and the $x$-axis can only  be $-\pi, 0, \pi$.  We claim that at $s_*$ the angle between $T(s_*)$ and the $x$-axis can only be zero, namely $x'(s_*)>0$, otherwise $c(s_*)$ must be on $x$-axis, namely $y(s_*)=0$, which contradicts to $y(s_*)<0$. If $s_*=s_j$ for some $s_j$, namely the tangent has a turn at $s_*$, then $y'(s_*-)\le0$ and $y'(s_*+)\ge0$ as the consequence of that $y(s_*)$ is the  negative minimum of $y(s)|_{[s', s'']}$. Since the angle between $T(s_*)$ and the $x$-axis must be in $(-\pi, \pi)$ (otherwise $y(s_*)=0$ as the above), and the angle between $T(s_*-)$ and the $x$-axis is in $(-\pi, 0]$, while the angle between $T(s_*+)$ and the $x$-axis is in $ [0, \pi)$, we can find a value between $T(s_*-)$ and $T(s_*+)$ such that it is  in the positive $x$-axis direction. Putting these together we have found $s_*\in [s', s'']$ and a vector  $\mathcal{N}\in T(s_*)$, which is a positive multiple of $c(s'')-c(s')$.

After a possible rotation, when $s_*$ is a smooth point,  we find linear isometric embedding $\iota: \mathbb{R}^2\to \mathbb{R}^{n+1}$ such that $\iota(\mathcal{N})=\tilde{T}(s_*)$. We will  specify the linear isometric embedding/identification later when $s_*$ is a point at which the tangent of $c(s)$ has a jump. Below we shall omit $\iota$ and denote the curve $\iota(c(s)$ ($\iota(T(s))$) as $c(s)$ ($T(s)$ respectively).

Now consider the two products $P_i$ defined as
$$
P_1:=\langle c(s'') -c(s'), \mathcal{N}\rangle= \int_{s'}^{s''} \langle T(s), \mathcal{N}\rangle\, ds,  \quad  P_2:=\langle \tilde{c}(s'') -\tilde{c}(s'), \mathcal{N}\rangle= \int_{s'}^{s''} \langle \tilde{T}(s), \mathcal{N}\rangle\, ds.
$$
From the choice of $s_*$ and $\mathcal{N}$,  $\langle c(s'') -c(s'), \mathcal{N}\rangle=|c(s'')-c(s')|$. Now $P_1=|c(s'')-c(s')|$, and $P_2=\langle \tilde{c}(s'') -\tilde{c}(s'), c(s'') -c(s')\rangle/|c(s'')-c(s')|$.

The claimed estimate (\ref{eq:11})  amounts to showing that the second product is bounded from below by the first (after a linear isometric embedding). Let $j_1$ be the smallest $j$ such that $s_j\ge s'$, $j_2$ be the biggest $j$ such that $s_j\le s''$ and $j_3$ be the biggest $j$ with $ s_j\le s_*$. Observe  the  convexity of $c(s)$ implies that for the case $s_*$ is a smooth point
\begin{eqnarray*}
&\,& \alpha_{j_0}+\sum_{j_1\le j\le j_3} \alpha_j+\int_{s'}^{s_{j_1}} k(s)\, ds+\sum_{j_1\le j\le j_3-1}\int_{s_j}^{s_{j+1}} k(s)\, ds+\int_{s_{j_3}}^{s_*} k(s)\, ds=\pi, \\
&\,&  \sum_{j_3+1\le j\le j_2} \alpha_j +\alpha_{j_{4}}+\int_{s_*}^{s_{j_3+1}} k(s)\, ds+\sum_{j_3+1\le j\le j_2-1}\int_{s_j}^{s_{j+1}} k(s)\, ds+\int_{s_{j_2}}^{s''} k(s)\, ds=\pi,
\end{eqnarray*}
with $\alpha_{j_0}$ being the angle from $-\mathcal{N}$ to  $T(s')$ and $\alpha_{j_4}$ being the angle from $T(s'')$ to  $-\mathcal{N}$. This implies that the image of $T([s', s_*])$ (a piece-wisely smooth spherical curve denoted by $\Gamma_1$) and the image of $T([s_*, s''])$ (another a piece-wisely smooth spherical curve denoted as $\Gamma_2$) are two minimizing arcs of the great circle. Geometrically, $\Gamma_1$ is the piecing together of $T|_{[s',s_{j_1})}$, the connecting  great arc jointing $T(s_{j_1}-)$ to $T(s_{j_1}+)$, $T|_{(s_j, s_{j+1})}$, for $j_1\le j\le j_3-1$, and connecting  arcs joining $T(s_j-)$ to $T(s_j+)$,  for $j_1+1\le j\le j_3$, and $T|_{(s_{j_3}, s_*]}$. One may express it in terms of the angle $\beta(\eta)$ from $\mathcal{N}$ and parametrize it using $\eta\in [s', s_*+\sum_{k=j_1}^{j_3} \alpha_k]$ as follows.
Let $\theta(s)$ denote the angle between $\mathcal{N}$ and  $T(s)$  when $s$ smooth (which is in $[-\pi, 0]$ for $s\in [s', s_*]$ and in $[0, \pi]$ for $s\in [s_*, s'']$). Here an orientation is picked for $c(s)\in \mathbb{R}^2$, as specified in the introduction.   For $s_j$, $\theta(s_j-)$ denotes the angle between $\mathcal{N}$ to $T(s_j-)$. Define $\theta(s_j+)$ similarly. Then $\theta(s_j+)-\theta(s_j-)=\alpha_j$. Explicitly   $\beta(\eta)$, for  $\eta\in [s', s_*+\sum_{j_1}^{j_3} \alpha_k]$,  which takes value in $[-\pi, 0]$,  can be given by
\begin{eqnarray*}
&\,& \beta(\eta)=\theta(\eta), \mbox{ for } s'\le \eta\le s_{j_1}; \quad \beta(s_{j_1})=\theta(s_{j_1}-); \\ &\,&\beta(\eta)=\theta(s_{j_1}-)+\eta-s_{j_1}, \mbox{ for } s_{j_1}<\eta<s_{j_1}+\alpha_{j_1}; \quad \beta(s_{j_1}+\alpha_{j_1})=\theta(s_{j_1}+);\\
&\,& \beta(\eta)=\theta(\eta-\alpha_{j_1}), \mbox{ for } s_{j_1}+\alpha_{j_1}<\eta<s_{j_1+1}+\alpha_{j_1}; \quad \beta(s_{j_1+1}+\alpha_{j_1})=\theta(s_{j_1+1}-).
\end{eqnarray*}
And for $j_3-1\ge j\ge j_1+1$, $\beta(s_j+\sum_{j_1}^{j-1} \alpha_k)=\theta(s_j-);$
\begin{eqnarray*}
&\,&  \beta(\eta)=\theta(s_j-)+\eta- (s_j+\sum_{j_1}^{j-1} \alpha_k), \mbox{ for }\eta\in (s_j+\sum_{j_1}^{j-1} \alpha_k,  s_j+ \sum_{j_1}^j \alpha_k);\\
&\,& \beta(s_j+\sum_{j_1}^j \alpha_k)=\theta(s_j+); \\
&\,& \beta(\eta)=\theta(\eta-\sum_{j_1}^j \alpha_k), \mbox{ for } \eta \in (s_j +\sum_{j_1}^j \alpha_k, s_{j+1}+\sum_{j_1}^{j}\alpha _k).
\end{eqnarray*}
And finally $\beta(s_{j_3}+\sum_{j_1}^{j_3-1}\alpha_k)=\theta(s_{j_3}-);$
\begin{eqnarray*}
&\,&\beta(\eta)=\theta(s_{j_3}-)+\eta-(s_{j_3}+\sum_{j_1}^{j_3-1} \alpha_k), \mbox{ for } \eta\in (s_{j_3}+\sum_{j_1}^{j_3-1} \alpha_k, s_{j_3}+\sum_{j_1}^{j_3} \alpha_k);\\
&\,& \beta(s_{j_3}+\sum_{j_1}^{j_3} \alpha_k)=\theta(s_{j_3}+); \\
&\,& \beta(\eta)=\theta (\eta-\sum_{j_1}^{j_3} \alpha_k), \mbox{ for } \eta \in (s_{j_3}+\sum_{j_1}^{j_3} \alpha_k, s_*+\sum_{j_1}^{j_3} \alpha_k].
\end{eqnarray*}
One can check that $\Gamma_1(\eta)$ is piece-wisely smooth. Moreover the convexity of $c(s)$ implies that $\beta(\eta'')\ge \beta(\eta')$ if $\eta''\ge \eta'$. Similarly a parametrization can be given  for $\Gamma_2$ with the corresponding $\beta(\eta)$ valuing in $[0, \pi]$. The choice of $\mathcal{N}$, and that there is no `folding' for $\Gamma_i$,  imply that
\begin{equation}\label{eq:key1}
\max\{\operatorname{Length} (\Gamma_1), \operatorname{Length} (\Gamma_2)\}\le \pi.
\end{equation}

We also denote the spherical curves corresponding to $\tilde{T}$ by $\widetilde{\Gamma}_i$. For example, $\widetilde{\Gamma}_1$ is obtained by piecing together of $\tilde{T}|_{[s',s_{j_1})}$, the minimizing  great arc  jointing $\tilde{T}(s_{j_1}-)$ to $\tilde{T}(s_{j_1}+)$, $\tilde{T}|_{(s_j, s_{j+1})}$,   for $j_1\le j\le j_3-1$, and minimizing  arcs joining $\tilde{T}(s_j-)$ to $\tilde{T}(s_j+)$,  for $j_1+1\le j\le j_3$, and $\tilde{T}|_{(s_{j_3}, s_*]}$.   However since $\tilde{c}(s)$ is not necessarily a convex plane curve, $\widetilde{\Gamma}_i$ over $(s_j, s_{j+1})$, namely $\tilde{T}|_{(s_j, s_{j+1})}$ is not necessarily in a plane. 
 Noting that $\mathcal{N}=T(s_*)=\tilde{T}(s_*)$ we estimate
\begin{eqnarray*}
\pi &\ge&  d_{\mathbb{S}^n}(T(s'), \mathcal{N})=d_{\mathbb{S}^n}(T(s'), T(s_*))=\operatorname{Length} (\Gamma_1)\\&=&\sum_{j_1\le j\le j_3} \alpha_j+\int_{s'}^{s_{j_1}} k(s)\, ds +\sum_{j_1\le j\le j_3-1}\int_{s_j}^{s_{j+1}} k(s)\, ds +\int_{s_{j_3}}^{s_*} k(s)\, ds\\
&\ge&\sum_{j_1\le j\le j_3} \tilde{\alpha}_j+\int_{s'}^{s_{j_1}} \tilde{k}(s)\, ds+\sum_{j_1\le j\le j_3-1}\int_{s_j}^{s_{j+1}} \tilde{k}(s)\, ds+\int_{s_{j_3}}^{s_*} \tilde{k}(s)\, ds\\
&=& \sum_{j_1\le j\le j_3} \tilde{\alpha}_j+\int_{s'}^{s_{j_1}}|\tilde{T}'|(s)\, ds + \sum_{j_1\le j\le j_3-1}\int_{s_j}^{s_{j+1}} |\tilde{T}'|(s)\, ds+\int_{s_{j_3}}^{s''}|\tilde{T}'|(s)\, ds\\
&=&  \operatorname{Length} (\widetilde{\Gamma}_1) \ge d_{\mathbb{S}^n}(\tilde{T}(s'), \tilde{T}(s_*))=d_{\mathbb{S}^n}(\tilde{T}(s'),\mathcal{N}).
\end{eqnarray*}
The second line above follows from the definition of the curvature (for the space curve) as the derivative of the  angle $\theta(s)$ between the tangent $T(s)$ and $T(s+h)$ for a curve \cite{Pogo} (cf. page 49) and that for a convex curve $k(s)\ge 0$. The third line uses the assumption of the theorem, and the last line follows from the definition of the (spherical) distance between two points being the infimum of the length of all possible piece-wisely smooth pathes (in $\mathbb{S}^n$) connecting them.  By the monotonicity of the parametrization of $\beta(\eta)$  the same argument also implies that for any smooth $s\in [s', s_*]$
$$
\pi\ge d_{\mathbb{S}^n}(T(s), \mathcal{N})=\operatorname{Length} (\Gamma_1|_{[s, s_*]})\ge \operatorname{Length} (\widetilde{\Gamma}_1|_{[s, s_*]}) \ge d_{\mathbb{S}^n}(\tilde{T}(s), \mathcal{N}).
$$
Here $\Gamma_1|_{[s, s_*]}$ means the curve obtained by piecing together $T|_{[s, s_*]}$ as explained above. The same applies to $\widetilde{\Gamma}_1|_{[s, s_*]}$.
The above estimate implies that  \begin{eqnarray}\label{eq:21}\langle T(s), \mathcal{N}\rangle &=&\cos ( d_{\mathbb{S}^n}( T(s), T(s_*))\le \cos ( d_{\mathbb{S}^n}( \tilde{T}(s), \tilde{T}(s_*))\\
&=&\cos ( d_{\mathbb{S}^n}( \tilde{T}(s), \mathcal{N}))=\langle \tilde{T}(s),\mathcal{N}\rangle.\nonumber\end{eqnarray}
Rewriting the above estimate we have that
$
\langle \tilde{T}(s)-T(s), \mathcal{N}\rangle \ge 0
$ for $s\in [s', s_*]$, which implies (\ref{eq:10}). A similar argument proves that the  inequality (\ref{eq:21})  holds also for $s\in [s_*, s'']$. Putting them together we have (\ref{eq:10}).

Now we explain how to handle the case when $s_*=s_{j_3}$ is not a smooth point. By the previous construction $\mathcal{N}\in T(s_{j_3})$. Recall that the image of $T(s_{j_3})$ is  the minimizing arc joining $T(s_{j_3}-)$ to $T(s_{j_3}+)$. Assume that $\gamma(t): [0, \alpha_{j_3}]\to \mathbb{S}^1$ is this arc parametrized by the arc-length $t$ with $\gamma(0)=T(s_{j_3}-)$ and $\gamma(\alpha_{j_3})=T(s_{j_3}+)$.  Assume that $\gamma(t_1)=\mathcal{N}$ for some $t_1\in [0, \alpha_{j_3}]$. There exists a similar minimizing geodesic $\tilde{\gamma}(t):[0, \tilde{\alpha}_{j_3}]\to \mathbb{S}^n$ corresponding to $\tilde{T}(s_{j_3})$. If $t_1 \le\tilde{\alpha}_{j_3}$, we choose a linear isometric embedding which identifies $\tilde{\gamma}(t_1)$ with $\mathcal{N}$. If $\alpha_{j_3}-t_1\le \tilde{\alpha}_{j_3}$ we identify $\mathcal{N}$ with $\tilde{\gamma}(\tilde{\alpha}_{j_3}-\alpha_{j_3}+t_1)$. If none of the previous two cases holds, namely $t_1 >\tilde{\alpha}_{j_3}$ and $\alpha_{j_3}-t_1> \tilde{\alpha}_{j_3}$,  we  identify $\mathcal{N}$ with any arbitrary chosen $\tilde{\gamma}(t)$ with $t\in [0, \tilde{\alpha}_{j_3}]$. For example one choice of the linear isometric embedding is to identify $\mathcal{N}$ with  $\tilde{\gamma}(0)=\tilde{T}(s_{j_3}-)$.  The above identification ensures  the comparison since
\begin{eqnarray*}
 d_{\mathbb{S}^n}(T(s'), \mathcal{N})&=& \sum_{j_1}^{j_3-1} \alpha_j+\int_{s'}^{s_{j_1}} k(s)\, ds+\sum_{j_1}^{ j_3-1}\int_{s_j}^{s_{j+1}} k(s)\, ds+ t_1\\
 &\ge&  \sum_{j_1}^{ j_3-1} \tilde{\alpha}_j+\int_{s'}^{s_{j_1}} \tilde{k}(s)\, ds+\sum_{j_1}^{ j_3-1}\int_{s_j}^{s_{j+1}} \tilde{k}(s)\, ds+d_{\mathbb{S}^n}(\tilde{T}(s_{j_3}-), \mathcal{N})\\
 & \ge& d_{\mathbb{S}^n}(\tilde{T}(s'), \mathcal{N})
\end{eqnarray*}
by $t_1\ge d_{\mathbb{S}^n}(\tilde{T}(s_{j_3}-), \mathcal{N})$, which equals to  either $t_1$,  $\tilde{\alpha}_{j_3}-\alpha_{j_3}+t_1$ or $0$ according to the three possible identifications above. (If in the last case of the above three situations we choose a linear isometric embedding to identify $\mathcal{N}$ with $\tilde{\gamma}(t_2)$ the arc jointing $\tilde{T}(s_{j_3}-)$ with $\tilde{T}(s_{j_3}+)$,   the estimate remains true due to that $t_1 >\tilde{\alpha}_{j_3}\ge t_2$ since this inequality holds when the first two situations do not arise.) The comparisons for $s\in [s', s_*]$ and $s\in [s_*, s'']$ are the same.

Now we compare the two products $P_i$ by writing
$$
P_1 =\int_{s'}^{s''} \langle T(s), \mathcal{N}\rangle \, ds =\left(\int_{s'}^{s_*}+\int_{s_*}^{s''}\right)\cos\left(d_{\mathbb{S}^n}( T(s), T(s_*)\right)\, ds.
$$
We express $P_2$ accordingly. The above estimate (\ref{eq:21}) implies that
\begin{equation}\label{eq:22}
\int_{\eta}^{s_*}\cos(d_{\mathbb{S}^n}( T(s), T(s_*))\, ds\le \int_{\eta}^{s_*}\cos(d_{\mathbb{S}^n}(\tilde{T}(s), \tilde{T}(s_*))\, ds, \quad \forall \eta \in [s', s_*]. 
\end{equation}
Similarly, we have
\begin{equation}\label{eq:23}
\int_{s_*}^{\eta}\cos(d_{\mathbb{S}^n}( T(s), T(s_*))\, ds\le \int_{s_*}^{\eta}\cos(d_{\mathbb{S}^n}(\tilde{T}(s), \tilde{T}(s_*))\, ds, \quad \forall \eta \in [s_*, s''].
\end{equation}
From (\ref{eq:22}) and (\ref{eq:23}), applied to $\eta=s'$ and $\eta=s''$ we have that $P_1\le P_2$, namely the  desired claim (\ref{eq:11}). The equality case can be seen by tracing  estimates. Applying them to $\eta=s'_*$ and $s''_*$ respectively we have (\ref{eq:101}). This completes the proof of corollary.

From the proof we have the following more general monotonicity since the proof for the monotonicity works as along as  $\mathcal{N}$ satisfies  (\ref{eq:key1}), while (\ref{eq:1-1}) implies that one can always choose a $s_*\in [0, L]$ and $\mathcal{N}\in T(s_*)$ independent of $s'$ and $s''$.

\begin{proposition}  Let $c: [0, L]\to \mathbb{R}^2$ be an embedded convex plane curve with curvature $k(s)\ge 0$. Let $\tilde{c}: [0, L]\to \mathbb{R}^{n+1}$ ($n\ge 1$) be a curve such that $\tilde{k}(s)\le k(s)$. Then there exists $s_*\in [0, L]$, $\mathcal{N}\in T(s_*)$ and a linear isometric embedding $\iota:\mathbb{R}^2\to \mathbb{R}^{n+1}$ with $\iota(0)=0$ and $\iota(\mathcal{N})\in \tilde{T}(s_*)$,  such that for any smooth $s$
 \begin{equation}\label{eq:24}
I_1'(s)= \langle \tilde{T}(s)-\iota(T(s)), \iota(\mathcal{N})\rangle \ge 0,\quad  \forall\, s \in [0, L], \mbox{ where } I_1(s)=\langle \tilde{c}(s)-\iota(c(s)), \iota(\mathcal{N})\rangle.
 \end{equation}
\end{proposition}

The proof can be easily adopted to show a comparison between a time-like curves in a Minkowski plane $L^2_1$  and another time-like curve in the three dimensional Minkowski space $L^3_1$ with signature $(+, -, -)$. In fact in terms of the monotonicity one may choose $s_*$ freely.

Following the convention of the physics  a vector $u$ is called time-like if $\langle u, u\rangle >0$. For a time-like curve $c(s)$, parametrized by the arc-length, $|T(s)|^2=|c'(s)|^2=1$. Hence $T(s)$ can be viewed as a point in the hyperbolic line $H^1$ (hyperbolic plane $H^2$) defined as $x_1^2-x_2^2=1$ ($x_1^2-x_2^2-x_3^2=1$, respectively). It can be checked easily that $-1$ multiple of the restricted metric on the line/surface is the standard hyperbolic metric. Using the arc-length $s$, $T(s)$ can be expressed as $(\cosh \theta(s), \sinh \theta(s))$. Hence $\theta(s)$ is the analogous angle function of $T(s)$ in $H^1$.  A simple computation shows that, for a convex curve, the angle difference $\varphi(s_2, s_1)=\theta(s_2)-\theta(s_1)$ equals to the hyperbolic distance between $T(s_1)$ and $T(s_2)$. In fact $T'(s)=(\sinh \theta(s), \cosh \theta(s))\theta'(s)$. In view that $d_{H^1}(T(s_2), T(s_1))=\int_{s_1}^{s_2}\sqrt{-\langle T'(s), T'(s)\rangle}=\int_{s_1}^{s_2} \theta'(s)\, ds=\theta(s_2)-\theta(s_1)$. Hence we used that $\theta'(s)=k(s)\ge 0$.
 This key observation is the hyperbolic analogue of the fact used in the above proof that the angle difference is the distance between two tangents in $\mathbb{S}^1$, which allows  a similar consideration as the above to yield the following result.

\begin{theorem}\label{thm:21}
 Let $c(s):[0, L]$ be a time-like convex curve in $L^2_1$ parametrized by the arc-length,  and let  $\tilde{c}(s):[0, L]$ be a similarly parametrized regular time-like curve in $L^3_1$. Assume that $k(s)\ge |\tilde{k}|(s)$. Then for any $s_*\in [0, L]$ and a linear isometric embedding of  $\iota :L^2_1\to  L^3_1$, which identifies $T(s_*)$ with $\tilde{T}(s_*)$,  we have that
\begin{equation}\label{eq:2add}
I_2'(s)= \langle \iota(T(s))-\tilde{T}(s), \tilde{T}(s_*)\rangle \ge 0, \mbox{ where } I_2(s)=\langle \iota (c(s))-\tilde{c}(s), \tilde{T}(s_*)\rangle.
\end{equation}
In particular, $|c(L)-c(0)|\ge |\tilde{c}(L)-\tilde{c}(0)|$. The equality holds if and only if $\iota(c(s))=\tilde{c}(s)$.
\end{theorem}
The last statement above generalizes the result of \cite{Lopez} by allowing the second curve $\tilde{c}(s)$ to be  a space curve in $L^3_1$. To derive this, by the first part of the argument in the proof of \ref{thm:main1} and Lemma 3 of \cite{Lopez}, we can find $s_*$ so that $T(s_*)$ (or a vector in $T(s_*)$ if $s_*$ is not a smooth point) is a positive multiple of $c(L)-c(0)$.  Now integrate (\ref{eq:2add}) with $s_*$  so chosen and then apply  the reserved Cauchy-Schwarz inequality (which holds for two time-like vectors in $L^3_1$). Note also that for the curves in two Minkowski planes, the result for space-like curves is the same as that for the time-like curves.

\section{Proof of Theorem \ref{thm:main2}}

We start with some basics on spherical (smooth) curves. Let $c(s)$ be a curve in $\mathbb{S}^2$ parametrized by the arc-length. Let $T(s)$ be its tangent, which is orthogonal to $c(s)$. Let $V(s)=c(s)\times T(s)$ be the cross product of $c(s)$ and $T(s)$ in $\mathbb{R}^3$, which is a normal of $c(s)$ in $T_{c(s)}\mathbb{S}^2$. The triple $\{c(s), T(s), V(s)\}$ forms an orthonormal moving frame (of $\mathbb{R}^3$) along $c(s)$. Since the geodesic curvature of a curve in the sphere (in a surface) is the changing rate of the tangential great circles (tangential geodesics in general,  by (8-3) of page 157 of \cite{Struik}), and that  $V(s)$ provides a natural parametrization of the tangential great circles, the derivative of $V(s)$ yields the geodesic curvature of $c(s)$. This can also be formulated  in terms of the following result.

\begin{proposition} \label{prop:31}
Let $k(s)$ be the geodesic curvature of $c(s)$ (with respect to $\mathbb{S}^2$). Then the following holds for $\{c(s), T(s), V(s)\}$.
\begin{eqnarray}
c'(s)&=&T(s),\label{eq:31}\\
T'(s)&=& k(s) V(s)-c(s), \label{eq:32}\\
V'(s)&=&-k(s)T(s). \label{eq:33}
\end{eqnarray}
\end{proposition}
\begin{proof} The first equation is definition. Also by definition $k(s)=\langle T'(s), V(s)\rangle$. Hence from
$$
0=\frac{d^2}{ds^2}\left( |c|^2(s)\right)=2\langle T(s), T(s)\rangle+2\langle c(s), T'(s)\rangle =2+2\langle c(s), T'(s)\rangle
$$
we deduce the second equation.
Now by the second equation
$$
V'(s)=T(s)\times T(s)+c(s)\times T'(s)=k(s)\, c(s)\times V(s)=-k(s) T(s).
$$
This prove the third one, hence completes the proof of the proposition.
\end{proof}
The local convexity of $c(s)$ is equivalent to $k(s)\ge0$.
The construction starts with  the cone $\mathcal{C}(c(s))$ over the spherical curve $c(s)$ centered at the origin,  and then obtain a plane curve $\mathcal{P}_c(s)$ by taking the intersection of $\mathcal{C}(c(s))$ with  a plane $P$ not passing the origin. There are many choices of such a plane. We simply choose one.  This curve can be expressed as $R(s) c(s)$ with $R(s)$ being the distance of  $\mathcal{P}_c(s)$ to the origin.  We need  the following formula for the curvature of the space curve in $\mathbb{R}^3$ applied to  $\mathcal{P}_c(s)$.

\begin{proposition}\label{prop:32} If $c(s)$ is a convex curve in $\mathbb{S}^2$, $\mathcal{P}_c(s)$ is a convex curve in $P$.  The curvature ${\bf k}(s)$ of  $\mathcal{P}_c(s)$ (as a space curve of $\mathbb{R}^3$) is given by
\begin{equation}\label{eq:34}
{\bf k}^2(s)=\frac{|\mathcal{P}'_c(s)\times \mathcal{P}''_c(s)|^2}{|\mathcal{P}'_c(s)|^{6}}.
\end{equation}
\end{proposition}
\begin{proof} From the geometric definition of the convexity we know that $c(s)$ lies in a signed semi-sphere cut out by any tangent great circle obtained by a plane $P_0$  passing the origin. Then it is clear that $\mathcal{P}_c(s)$ lies on the corresponding half plane cut out by the corresponding tangent line of $\mathcal{P}_c(s)$  in $P$ which is the intersection of $P_0$ and $P$. This proves the convexity of $\mathcal{P}_c(s)$. The formula for the curvature of a space curve is well known and computational. See for example page 51 of \cite{Pogo}. Of course the formula does apply to the case that the curve happens to be a plane curve.
\end{proof}

Now let $\tau$ be the arc-length parameter of $\mathcal{P}_c(s)$. Direct calculation shows that
\begin{equation}\label{eq:35}
\tau(s)=\int_0^s \sqrt{(R'(s))^2+R^2(s)}\, ds.
\end{equation}

Now we construct a  space curve $\tilde{\mathcal{P}}_{\tilde{c}}(s)$ corresponding to $\tilde{c}(s)$ by defining it as $R(s) \tilde{c}(s)$. In general, this is not a plane curve. The key observation is that the arc-length parameter for $\tilde{\mathcal{P}}_{\tilde{c}}(s)$ is the same as that of $\mathcal{P}_c(s)$, namely it is given by (\ref{eq:35}) as well,  since $|\tilde{c}|(s)=1=|c(s)|$ and $|\tilde{c}'|(s)=1=|c'(s)|$.
Moreover its curvature $\tilde{{\bf k}}(s)$ (as a curve in $\mathbb{R}^3$) can be expressed similarly as
\begin{equation}
\tilde{{\bf k}}^2(s)= \frac{|\tilde{\mathcal{P}}'_{\tilde{c}}(s)\times \tilde{\mathcal{P}}''_{\tilde{c}}(s)|^2}{|\tilde{\mathcal{P}}'_{\tilde{c}}(s)|^{6}}.
\end{equation}
Namely the second part of Proposition \ref{prop:32} applies to $\tilde{\mathcal{P}}_{\tilde{c}}(s)$ as well since it holds for any space curve in $\mathbb{R}^3$.
The key step is the following comparison.

\begin{proposition}\label{prop:33}  Under the assumption that the geodesic curvature $k(s)$ of $c(s)$ and the geodesic curvature $\tilde{k}(s)$ of $\tilde{c}(s)$ satisfy $k(s)\ge |\tilde{k}(s)|\ge 0$, the curvatures of $\mathcal{P}_c(s)$ and $\tilde{\mathcal{P}}_{\tilde{c}}(s)$ satisfy
${\bf k}(s)\ge 0$ and ${\bf k}(s)\ge |\tilde{{\bf k}}|(s).$
\end{proposition}
\begin{proof} Since $\mathcal{P}_c(s)$ is convex, we have that ${\bf k}(s)\ge 0$, it suffices to show that  ${\bf k}^2(s)\ge \tilde{{\bf k}}^2(s)$. First we observe that
$|\mathcal{P}'_c(s)|^2=R^2(s)+(R'(s))^2=|\tilde{\mathcal{P}}'_{\tilde{c}}(s)|^2$. This reduces the desired  estimate to
\begin{equation}\label{eq:de1}
|\mathcal{P}'_c(s)\times \mathcal{P}''_c(s)|^2\ge |\tilde{\mathcal{P}}'_{\tilde{c}}(s)\times \tilde{\mathcal{P}}''_{\tilde{c}}(s)|^2.
\end{equation}
Using the fact that $\{c(s), T(s), V(s)\}$ forms an oriented orthonormal moving frame, a direct calculation, using Proposition \ref{prop:31}, shows that
\begin{eqnarray*}
\mathcal{P}'_c(s)\times \mathcal{P}''_c(s)&=& (R'(s) c(s)+R(s) T(s))\times (R''(s) c(s) +2 R'(s) T(s) + R(s) T'(s))\\
&=& ( 2 (R'(s))^2 -R(s)R''(s))V(s)- R'(s) R(s) k(s) \, T(s)\\
&\,&+R^2(s) k(s) c(s) +R^2(s) V(s)\\
&=&R^2(s) k(s) c(s)- R'(s) R(s) k(s) \, T(s)\\
&\,&+(2 (R'(s))^2 -R(s)R''(s)+R^2(s))V(s).
\end{eqnarray*}
Hence we have that
\begin{eqnarray}
|\mathcal{P}'_c(s)\times \mathcal{P}''_c(s)|^2&=&(R^4(s)+(R'(s) R(s))^2)k^2(s)\nonumber\\
&\,&+(2 (R'(s))^2 -R(s)R''(s)+R^2(s))^2. \label{eq:37}
\end{eqnarray}
A similar calculation shows that
\begin{eqnarray}
|\tilde{\mathcal{P}}'_{\tilde{c}}(s)\times \tilde{\mathcal{P}}''_{\tilde{c}}(s)|^2&=&(R^4(s)+(R'(s) R(s))^2)\tilde{k}^2(s)\nonumber\\
&\,&+(2 (R'(s))^2 -R(s)R''(s)+R^2(s))^2. \label{eq:38}
\end{eqnarray}
From (\ref{eq:37}) and (\ref{eq:38}), the assumption $k(s)\ge |\tilde{k}|(s)$ implies (\ref{eq:de1}), hence the desired estimate of the proposition.
\end{proof}

Now Proposition \ref{prop:33} and (\ref{eq:35}) implies that $\mathcal{P}_c(\tau)$ and $\tilde{\mathcal{P}}_{\tilde{c}}(\tau)$ are two curves satisfying the assumption of Theorem \ref{thm:main1}. Hence we have that
$$
d_{\mathbb{R}^3}(\mathcal{P}_c(0), \mathcal{P}_c(\tau(L))) \le d_{\mathbb{R}^3}(\tilde{\mathcal{P}}_{\tilde{c}}(0), \tilde{\mathcal{P}}_{\tilde{c}}(\tau(L))).
$$
Theorem \ref{thm:main2} for the smooth curves now follows from the hinge theorem of the Euclidean geometry.

For the general case when the tangents of $c(s)$ and $\tilde{c}(s)$ have finite jumps at $\{s_j\}$, if we denote the turning angles at $\mathcal{P}_c(s_j)$ and  $\tilde{\mathcal{P}}_{\tilde{c}}(s_j)$ by $\theta_j$ and $\tilde{\theta}_j$, then
\begin{equation}\label{eq:310}
\cos \theta_j= \frac{ R'(s_j-) R'(s_j+) +R^2(s_j) \cos \alpha_j}{\sqrt{\left((R'(s_j-))^2+R^2(s_j)\right)\left((R'(s_j+))^2+R^2(s_j)\right)}}.
\end{equation}
By a similar formula for $\cos \tilde{\theta_j}$ we deduce that $\theta_j\ge \tilde{\theta}_j$ if $\alpha_j\ge \tilde{\alpha}_j$.
Hence Theorem \ref{thm:main2} follows from the general case of Theorem \ref{thm:main1}.

The argument can be modified to prove the following more general result.

\begin{theorem}\label{thm:32} Let $c(s): [0, L]\to \mathbb{S}^2$ (with $L\le \pi)$ be a simple piece wisely smooth curve in $\mathbb{S}^2$ parametrized by the arc length, such that its tangent has finite many discontinuities at $\{s_j\}_{j=1}^{N-1}$. Assume that the curve is convex. Namely with a positive orientation, the turning angle of the tangent at $s_j$ from $T(s_j-)$ to $T(s_j+)$ counter clock wisely is $\alpha_j\in (0, \pi)$ and the geodesic curvature $k(s)$ of $c(s)\in \mathbb{S}^2$ over smooth $s$ is nonnegative. Assume that $\tilde{c}(s):[0, L]\to \mathbb{S}^n$ is another simple piece wisely smooth curve satisfying that  the  only possible discontinuities of  its tangent $\tilde{T}(s)$ are at $\{s_j\}$, and the angle $\tilde{\alpha}_j$ between $T(s_j-)$ and $T(s_j+)$  satisfies $\tilde{\alpha}_j\le \alpha_j$. Moreover assume that the geodesic curvature $\tilde{k}(s)$ satisfies that $0\le \tilde{k}(s)\le k(s)$. Then we have that $d_{\mathbb{S}^2}(c(0), c(L))\le d_{\mathbb{S}^n}(\tilde{c}(0), \tilde{c}(L))$. The equality holds if and only if $\tilde{c}(s)$ is congruent to $c(s)$ after a proper linear isometric embedding of $\mathbb{S}^2$ into $\mathbb{S}^n$.
\end{theorem}

To prove this general result we need modified versions of Proposition \ref{prop:31}--\ref{prop:33} for the curves $\tilde{c}(s)$ in $\mathbb{S}^n$ and the corresponding curve $\tilde{\mathcal{P}}_{\tilde{c}}(s)= R(s)\cdot \tilde{c}(s)$ in  $\mathbb{R}^{n+1}$.

The modified Proposition \ref{prop:31}  for $\tilde{c}(s)\in \mathbb{S}^n$ amounts to properly defining $\tilde{k}(s)$. Let $\frac{D \tilde{T}}{ds} $ denote the covariant derivative of the tangent $\tilde{T}$ as a vector in $T_{\tilde{c}(s)}\mathbb{S}^n$. Then we have that
$$
\frac{D \tilde{T}}{ds}=\tilde{T}'(s)-\langle \tilde{T}'(s), \tilde{c}(s)\rangle=\tilde{T}'(s)+\tilde{c}(s).
$$
Here we also used
$$
0=\frac{1}{2}\frac{d^2}{ds^2}|\tilde{c}(s)|^2=\langle \tilde{T}'(s), \tilde{c}(s)\rangle+|\tilde{T}(s)|^2.
$$
 We define $\tilde{k}(s)=\left|\frac{D \tilde{T}}{ds}\right|$. When $\frac{D \tilde{T}}{ds}\ne 0$ we may write it as $\tilde{k}(s) \tilde{V}(s)$ with $\tilde{V}(s)$ being a unit vector which is normal to both $\tilde{c}(s)$, and $\tilde{T}(s)$. Given that $k(s)\ge 0$, the assumption $\tilde{k}(s)\le k(s)$ makes sense. With this convention we do have that
\begin{equation}\label{eq:311}
\tilde{c}'(s)= \tilde{T}(s);  \quad \tilde{T}'(s)=\tilde{k}(s) \tilde{V}(s)-\tilde{c}(s)
\end{equation}
which are sufficient for the comparison of the curvatures of the constructed curves $\mathcal{P}_c(s)$ and $\tilde{\mathcal{P}}_{\tilde{c}}(s)$ in $\mathbb{R}^3$ and $\mathbb{R}^{n+1}$. Note that $\tilde{V}(s)$ is only defined locally when $|\frac{D \tilde{T}}{ds}|(s)\ne 0$, and $\{\tilde{c}(s), \tilde{T}(s), \tilde{V}(s)\}$ are orthonormal when $\tilde{V}$ is defined.

The key is a generalization of Proposition \ref{prop:32} since for this general case we can not use a formula involving the cross product to compute the curvature of a curve in $\mathbb{R}^{n+1}$.

\begin{proposition}\label{prop:34} Let $r(s): (a, b)\to \mathbb{R}^{n+1}$ be a regular curve and let $\tau$ be the arc-length parameter.    Let ${\bf k}(s)$ be its curvature (defined as $\sqrt{\langle \frac{d^2 r}{d\tau^2} (\tau), \frac{d^2 r}{d\tau^2}(\tau)\rangle}$). Then
\begin{equation}\label{eq:312}
{\bf k}^2(s)=\frac{ \langle r''(s), r''(s)\rangle\langle r'(s), r'(s)\rangle-\langle r''(s), r'(s)\rangle^2}{\langle r'(s), r'(s)\rangle^3}.
\end{equation}
\end{proposition}
\begin{proof} Note that $\frac{d\tau}{ds}=|r'|$, $\frac{d r}{d\tau}=r' \frac{ds}{d\tau}=r'/|r'|$. Hence we have that
\begin{eqnarray*}
\frac{d^2 r}{d\tau^2}&=& \frac{d}{ds}\left(\frac{d r}{d\tau}\right) \frac{ds}{d\tau}= \frac{d}{ds}\left(r'/|r'|\right) |r'|^{-1}\\
&=&\frac{r''|r'|^2-\langle r'', r'\rangle r'}{\langle r', r'\rangle^{2}}.
\end{eqnarray*}
Using the definition of ${\bf k}(s)$ we have that
\begin{eqnarray*}
{\bf k}^2(s)&=&\langle \frac{d^2 r}{d\tau^2}, \frac{d^2 r}{d\tau^2}\rangle \\
&=& \frac{|r''|^2|r'|^4-\langle r'', r'\rangle^2 |r'|^2}{|r'|^8}
=\frac{|r''|^2|r'|^2-\langle r'', r'\rangle^2}{\langle r', r'\rangle^{3}}.
\end{eqnarray*}
This proves the formula claimed.
\end{proof}

Now we prove a similar comparison result as Proposition \ref{prop:33}.

 \begin{proposition}\label{prop:35}  Under the assumption that the geodesic curvature $k(s)$ of $c(s)$ and the geodesic curvature $\tilde{k}(s)$ of $\tilde{c}(s)$ satisfy $k(s)\ge \tilde{k}(s)\ge 0$, the curvatures of $\mathcal{P}_c(s)$ and $\tilde{\mathcal{P}}_{\tilde{c}}(s)$ (as curves in $\mathbb{R}^2$ and in $\mathbb{R}^{n+1}$)  satisfy
${\bf k}(s)\ge 0$ and ${\bf k}(s)\ge |\tilde{{\bf k}}|(s).$
\end{proposition}
\begin{proof}
The only difference is we use equation (\ref{eq:312}) to compute the curvature  $ |\tilde{{\bf k}}|^2(s)$ of $\tilde{\mathcal{P}}_{\tilde{c}}(s)$.
To simplify the notation we denote $\tilde{\mathcal{P}}_{\tilde{c}}(s)$ by $r(s)$, which is given by $R(s)\cdot \tilde{c}(s)$.  Clearly $r'(s)=R'(s) \tilde{c}(s)+R(s) \tilde{c}'(s)=R'(s) \tilde{c}(s)+R(s) \tilde{T}(s)$, and by (\ref{eq:311})
\begin{equation}\label{eq:313}
r''=(R''-R ) \tilde{c} +2R' \tilde{T}+R \tilde{k} \tilde{V}.
\end{equation}
Note that the above also holds when $\tilde{k}=0$. Now direct calculation shows
\begin{eqnarray*}
&\,&\langle r'', r''\rangle=(R''-R)^2 +4(R')^2 +R^2 \tilde{k}^2;\\
&\,& \langle r', r'\rangle =(R')^2+R^2;\\
&\,& \langle r'', r'\rangle^2=\left( (R''-R)R'+2R'R\right)^2.
\end{eqnarray*}
Putting the above  together,  Proposition \ref{prop:34} implies that
$$
\left((R')^2+R^2\right)^3 |\tilde{{\bf k}}|^2(s)=\left( (R''-R)R-2(R')^2\right)^2 +R^2 \tilde{k}^2((R')^2+R^2).
$$
This together with (\ref{eq:34}), the fact that $|\mathcal{P}'_c(s)|^2=R^2(s)+(R'(s))^2$, and (\ref{eq:37}) proves the claim.
\end{proof}

Now Theorem \ref{thm:32} follows from Theorem \ref{thm:main1} (precisely (\ref{eq:11}) applied to $\mathcal{P}_c(s)$ and $\tilde{\mathcal{P}}_{\tilde{c}}(s)$), Proposition \ref{prop:35}, (\ref{eq:310}) and the argument  of proving Theorem \ref{thm:main2}.

 \section*{Acknowledgments} {} The author would like to thank Burkhard Wilking for helpful discussions, Paul Bryan, Jon Wolfson, H. Wu and Fangyang Zheng for their interests to the problem considered. He is also grateful to a  referee who spotted a discrepancy in an earlier version of the paper. The author also thanks Pak-Yeung Chan who read a later version and made some  comments which helped the presentation.


\begin{thebibliography}{A}






\bibitem{Alex} A. D. Alexandrow,\textit{ Konvexe Polyeder.} German translation from Russian; Akademie-Verlag,
Berlin, 1958.


\bibitem{An} B. Andrews and J.  Clutterbuck, \textit{Proof of the fundamental gap conjecture.}  J. Amer. Math. Soc. \textbf{24} (2011), no. 3, 899--916.

\bibitem{Arnold} V.I. Arnold, \textit{ The geometry of spherical curves and the algebra of quaternions.} Uspekhi Mat. Nauk \textbf{50}(1995), 3--68.

\bibitem{Busemann} H. Busemann, \textit{Convex surfaces.} Interscience Tracts in Pure and Applied Mathematics, no. {\bf 6}. Interscience Publishers, Inc., New York; Interscience Publishers Ltd., London 1958 ix+196 pp.

\bibitem{Cauchy} A. Cauchy, \textit{Sur les polygones et les poly\`edres.} Second M\'emoire, Oeuvres Compl\`etes, II$^e$
S\'erie, vol. 1; Paris, 1905.


\bibitem{Chern} S. S. Chern, \textit{ Curves and surfaces in Euclidean space.} 1967 Studies in Global Geometry and Analysis pp. 16--56 Math. Assoc. America, Buffalo, N.Y.; distributed by Prentice-Hall, Englewood Cliffs, N.J.

\bibitem{Epstein} C. L. Epstein, \textit{The theorem of A Schur in hyperbolic spaces.}  Preprint 46 pages, 1985.


\bibitem{Hopf} H. Hopf, \textit{Differential Geometry in the Large.} Notes taken by Peter Lax and John Gray. With a preface by S. S. Chern. Second edition. With a preface by K. Voss. Lecture Notes in Mathematics, 1000. Springer-Verlag, Berlin, 1989. viii+184 pp.

\bibitem{Lopez} R. L\'opez, \textit{The theorem of Schur in the Minkowski plane.} Jour. Geom. Phys. \textbf{61} (2011), 342--346.


\bibitem{Ni-JMPA} L. Ni, \textit{Estimates on the modulus of expansion for vector fields solving nonlinear equations.} J. Math. Pures Appl. (9) \textbf{99} (2013), no. 1, 1--16.


\bibitem{Pogo} A. V. Pogorelov, \textit{Differential Geometry.} P. Noodhoff N. V. 1960.




\bibitem{SR} E. Steinitz and H. Rademacher, \textit{Vorlesungen \"uber die Th\'eorie der Polyeder.} Springer-Verlag,
Berlin, 1934.


\bibitem{SZ}  I. J. Schoenberg and  S. C. Zaremba, \textit{
On Cauchy's lemma concerning convex polygons.}
Canadian J. Math. \textbf{19} (1967), 1062--1071.

\bibitem{Struik} D. J. Struik, \textit{Lectures on Classical Differential Geometry.} 2nd Edition, Dover, 1988.


\end{thebibliography}
\end{document}